\documentclass[12pt]{amsart}
\usepackage[ansinew]{inputenc}

\usepackage{a4, amsmath,amssymb,amsbsy,amsfonts,amsthm,latexsym,
            amsopn,amstext,amsxtra,euscript,amscd,
graphicx,psfrag, epsfig,color,multirow}

\usepackage[english]{babel}
\usepackage{url}
\usepackage[colorlinks,linkcolor=blue,anchorcolor=blue,citecolor=blue]{hyperref}

\newtheorem{thm}{Theorem}

\newtheorem{lemma}[thm]{Lemma}

\newtheorem{theorem}[thm]{Theorem}
\newtheorem{cor}[thm]{Corollary}

\numberwithin{equation}{section}
\numberwithin{thm}{section}

\def\\{\cr}
\def\({\left(}
\def\){\right)}
\def\[{\left[}
\def\]{\right]}
\def\<{\langle}
\def\>{\rangle}
\def\fl#1{\left\lfloor#1\right\rfloor}
\def\rf#1{\left\lceil#1\right\rceil}

\def\cD{\mathcal D}

\def\F{\mathbb{F}}

\def\Z{\mathbb{Z}}

\def\cA{{\mathcal A}}
\def\cB{{\mathcal B}}

\def\cD{{\mathcal D}}

\def\cH{{\mathcal H}}
\def\cI{{\mathcal I}}
\def\cJ{{\mathcal J}}

\def\cN{{\mathcal N}}

\def\cS{{\mathcal S}}
\def\cT{{\mathcal T}}
\def\cU{{\mathcal U}}
\def\cV{{\mathcal V}}
\def\cW{{\mathcal W}}

\def\vec#1{\mathbf{#1}}

\def\fA{{\mathfrak A}}

\def\free{\ast}

\def\mand{\qquad \mbox{and} \qquad}

\begin{document}

\title[Binary digits of squarefree numbers and quadratic residues]
{Prescribing the binary digits of squarefree numbers and quadratic residues}

\author[R. Dietmann] {Rainer Dietmann}
\address{Department of Mathematics, Royal Holloway,
University of London, Egham, Surrey, TW20 0EX, United Kingdom}
\email{rainer.dietmann@rhul.ac.uk}

\author[C. Elsholtz]{Christian Elsholtz}
\address{Institute of Analysis and Computational Number Theory, Technische  Universit\"at  Graz, A-8010 Graz, Austria }  
 \email{elsholtz@math.tugraz.at}

\author[I. E. Shparlinski]{Igor E. Shparlinski} 
\address{Department of Pure Mathematics, University of New South Wales, 
Sydney, NSW 2052, Australia}
\email{igor.shparlinski@unsw.edu.au}
\begin{abstract} 
We study the equidistribution of
multiplicatively defined sets, such as the squarefree integers,
quadratic non-residues or primitive roots, in sets
which are described in an additive way, such as sumsets
or Hilbert cubes. 
In particular, we show that  if one  fixes
any proportion  less than  $40\%$
of the digits of all numbers of a given binary bit length, then the remaining set
still has the asymptotically  expected number
of squarefree integers.
Next, we investigate the distribution of primitive roots
modulo a large prime $p$, establishing a new upper bound on
the largest dimension of a Hilbert
cube in the set of primitive roots, improving on a previous result
of the authors.
Finally, we study sumsets in finite fields and asymptotically find
the expected number of quadratic residues and non-residues in
such sumsets, given their cardinalities are big enough. This 
significantly improves 
on a recent result by Dartyge, Mauduit and S{\'a}rk{\"o}zy.
Our approach introduces several new ideas, combining
a variety of methods, such as bounds of exponential and
character sums, geometry of numbers and additive combinatorics.
\end{abstract}


\maketitle

\section{Introduction}

\subsection{Motivation}
Given  an integer $n  \ge 2$,
we denote by 
$\cD_{n}$ the set of vectors $\vec{d} = (\delta_0, \ldots, \delta_{n-1})$
where $\delta_i \in \{\free, 0, 1\}$, $i =0, \ldots, n-1$, 
and for $\vec{d} \in \cD_n$
we consider 
the set 
$$
\cN_n(\vec{d}) = \left\{\sum_{i=0}^{n-1} d_i 2^i~:~ 
d_i \in \{0, 1\}\ \text{if}\ \delta_i=\free, 
\ d_i =\delta_i\ \text{otherwise}\right\}.
$$

Furthermore, for  integers $n >k \ge 1$ we denote 
by $\cD_{k,n}$ the subset of $n$-dimensional  vectors $\vec{d} \in \cD_n$
with exactly $k$ components of $\vec{d}$  that are fixed as either
$0$ or $1$ and $n-k$ components that are $\free$. We also use $\cD_{k,n}^*$ to denote the set 
of vectors  $\vec{d} \in \cD_{k,n}$ with $\delta_0 = 1$.
In particular, for $\vec{d} \in \cD_{k,n}^*$ all elements 
of $\cN_n(\vec{d})$ are odd. 

Various arithmetic properties of elements from $\cN_n(\vec{d})$
as well as of other integers with restricted digits 
have been studied in a number of works. 

We first recall that Bourgain~\cite{Bour1,Bour2}
has recently obtained 
several very strong results about prime numbers with 
prescribed binary digits, see also~\cite{HaKa}. For example, 
the result of~\cite{Bour2}
gives an asymptotic formula for the number of primes $p \in \cN_n(\vec{d})$
for very dense vectors  $\vec{d} \in  \cD_{k,n}^*$, more precisely, 
when $k \le cn$ for some absolute constant $c>0$, which is 
a dramatic improvement over the previous results of~\cite{Bour1,HaKa}. 
Mauduit and Rivat~\cite{MauRiv} have recently settled a 
problem of Gelfond about the distribution of primes
with the sums of digits in a prescribed arithmetic progression,  
see also~\cite{DMR}. Partially motivated by some cryptographic applications, 
the distribution and construction of 
RSA moduli and smooth numbers with some binary digits prescribed 
have been studied in~\cite{GraShp,Shp2}. Various results 
on prime divisors 
and other arithmetic properties of numbers with very few non-zero
binary digits can be found in~\cite{BaSh1,Bour0,EMS,Luca,Shp1,Shp3}. 
For a diverse variety of results on integers with various restrictions on their 
digits (for example, palindromes)
see~\cite{BaSh2,Col1,Col2,DrMa,Kon1,KMS,MaSha,MoShk} and references therein.

\subsection{Our results and methods} 
Here we combine a variety of methods, such as bounds on exponential and 
character sums, geometry of numbers, additive combinatorics, 
to derive new results about the arithmetic structure of elements 
of $\cN_n(\vec{d})$. Throughout, our goal is to treat $\vec{d} \in \cD_{k,n}$ with 
the ratio
$k/n$ as large as possible (that is, for thin sets of integers with
as large as possible proportion of pre-assigned digits). We believe that 
our ideas and  results can find application to several other problems
of this type. 

More precisely, in Section~\ref{sec:sqfr} we first study
the distribution of squarefree numbers in  sets
$\cN_n(\vec{d})$.
 Using some combinatorial arguments, the theory of lattice minima 
and a result of Bourgain~\cite[Lemma~4]{Bour2} we obtain an 
asymptotic formula for the number of squarefree integers 
$s \in \cN_n(\vec{d})$ for $\vec{d} \in  \cD_{k,n}^*$ provided that  
$k\le (2/5- \varepsilon) n$ for any fixed  $\varepsilon>0$. 
In Section~\ref{sec:Euler} we also give an asymptotic for the sums of values of the Euler function 
in essentially full range $k\le (1- \varepsilon) n$. 

Furthermore, we also estimate multiplicative character sums and
in  Section~\ref{sec:nonres} obtain 
results about the distribution of quadratic non-residues and primitive roots 
modulo $p$ among the elements of $\cN_n(\vec{d})$  for 
$\vec{d} \in  \cD_{k,n}$ provided that 
$k\le (1/2 - \varepsilon) n$ for any fixed  $\varepsilon>0$. 
This result complements those of~\cite{BaCoSh,DES1,DES2,OstShp}
where similar questions are considered for integers with various restrictions
on their binary digits (and also digits in other bases). 

Finally,   in Section~\ref{sec:Hilb}, we  consider a related question about quadratic residues and primitive roots in  {\it Hilbert cubes\/}.
For a prime power $r = p^n$, let $\F_r$ denote the finite field of $r$ elements.
For $a_0, a_1,
\ldots, a_d \in \F_p$ we define the {\it Hilbert cube\/} as
\begin{equation}
\label{eq:hilbert_cube}
\cH(a_0; a_1, \ldots, a_d) = \left\{ a_0 + \sum_{i=1}^d \vartheta_i a_i~:~
\vartheta_i \in \{0,1\}\right\}.
\end{equation}

We define $f(p)$ as the largest $d$ such that 
there are $a_0, a_1, \ldots, a_d \in \F_p$ with   pairwise distinct
$a_1, \ldots, a_d$ such that
$\cH(a_0; a_1, \ldots, a_d)$ does not contain a quadratic non-residue 
modulo $p$. Furthermore, we define $F(p)$ 
as the largest $d$ 
 such that there are $a_0, a_1, \ldots, a_d \in \F_p$ with pairwise distinct
$a_1, \ldots, a_d$ such that $\cH(a_0; a_1, \ldots, a_d)$ does not contain a 
primitive root modulo $p$. Clearly
$$
 f(p) \le F(p).
$$
 Hegyv\'{a}ri and S\'{a}rk\"{o}zy~\cite[Theorem~2]{HS} give the bound 
 $f(p) < 12 p^{1/4}$ for sufficiently large $p$, which has been improved
to
$$
F(p) \le p^{1/5+o(1)}
$$
as $p\to \infty$, by Dietmann, Elsholtz and Shparlinski~\cite[Theorem~1.3]{DES1}. 
Here we improve this further to 
$$
F(p) \le p^{3/19+o(1)}
$$
and recall that reducing the exponent below $1/8$
immediately implies an  improvement of the Burgess bound on the least primitive root
(note that $3/19 -1/8 = 0.0328\ldots$). 

As a further application of our method of Section~\ref{sec:nonres}, 
in Section~\ref{sec:rest dig},  
we outline a substantial  improvement of one of the results of
Dartyge, Mauduit and S{\'a}rk{\"o}zy~\cite{DMS}. Namely, 
given a basis $\omega_1, \ldots, \omega_n$ of $\F_{p^n}$ 
over $\F_p$, and  a collection of $n$ sets 
$\fA = \{\cA_i\subseteq \F_p~:~i=1, \ldots, n\}$,
we consider the set 
\begin{equation}
\label{eq:Set WA}
\cW_\fA = \left\{a_1 \omega_1+\ldots+a_n \omega_n~:~ 
a_i \in \cA_i, \ i=1, \ldots, n\right\},
\end{equation}
which has a natural interpretation of elements in $\F_{p^n}$ 
``with restricted digits''. 
Dartyge,  Mauduit  and  S{\'a}rk{\"o}zy~\cite[Theorem~2.1]{DMS}
show that if for some fixed $\varepsilon > 0$ the lower bound
\begin{equation}
\label{eq:DMS Cond}
\min_{1 \le i \le r}\# \cA_i  \ge \(\frac{ \sqrt{5}-1}{2} + \varepsilon\)p
\end{equation}
holds, then, as $p\to \infty$,
the set $\cW_\fA$ contains asymptotically equal proportions 
of quadratic residues and non-residues  (note that in~\cite{DMS}
only the case of $\cA_1 =\ldots =\cA_r$ is considered but the 
proof immediately extends to different sets).

Here, in Section~\ref{sec:rest dig} we prove a similar asymptotic equidistribution 
of quadratic residues and non-residues under a much more
relaxed condition than~\eqref{eq:DMS Cond}. Namely, for our result 
we only assume that for some fixed $\varepsilon > 0$ we have 
$$
\prod_{1 \le i \le n}\# \cA_i  \ge  p^{(1/2 + \varepsilon)n^2/(n-1)}
\mand
\min_{1 \le i \le n}\# \cA_i  \ge  p^\varepsilon.
$$
For $n \geq 3$ this is a much wider range of parameters than the
earlier restriction~\eqref{eq:DMS Cond} that is linear in $p$.

\subsection{Notation} 
Throughout the paper the implied constants in the symbols ``$O$'' and
``$\ll$'' may depend on the real parameter $\varepsilon > 0$ and
 an integer parameter $\nu \ge 1$.  We recall that
the expressions $A \ll B$ and $A=O(B)$ are each equivalent to the
statement that $|A|\le cB$ for some constant $c$.  

As usual,
$\log z$ denotes the natural logarithm of $z$.

The letter $p$ always denotes a prime.

As we have mentioned, we use $\F_r$ to denote the finite field of $r$ elements.

\section{Preparations}
\subsection{Bounds of some exponential and character sums}
\label{sec:exp char}

We need the following result of Bourgain~\cite[Lemma~4]{Bour2}.

\begin{lemma}
\label{lem:ExpSums} Let $n > k \ge 1$ and
$\vec{d} \in \cD_{k,n}^*$. Then for any 
integers $a$ and $q$ with $\gcd(2a,q) = 1$ and 
 $3 \le q \le n^{1/10  \kappa}$, where $\kappa = k/n$,  we have
 $$
\left|\sum_{s \in \cN_n(\vec{d})} \exp(2 \pi i as/q)\right|
 < \#\cN_n(\vec{d}) 2^{-\sqrt{n}}.
$$
\end{lemma}

We need the following bound of a double character sum
due to Karatsuba~\cite{Kar1}, see also~\cite[Chapter~VIII, Problem~9]{Kar2}, 
which in turn follows from the
Weil bound (see~\cite[Corollary~11.24]{IwKow}) and the H{\"o}lder inequality.

We present it in the settings of arbitrary finite fields.

\begin{lemma}
\label{lem:DoubleSums} 
For any integer $\nu \ge 1$, any sets
$\cA, \cB \subseteq \F_r$ and any non-trivial multiplicative character  $\chi$ 
of $\F_r$,  
we have
$$
\sum_{a\in \cA}\sum_{b \in \cB} \chi(a+b)  \ll  (\# \cA)^{1-1/2\nu} \#\cB r^{1/4\nu} +  
(\# \cA)^{1-1/2\nu} (\#\cB)^{1/2} r^{1/2\nu} ,
$$
where the implied constant depends only on $\nu$. 
\end{lemma}

In particular, taking $\nu =\rf{\varepsilon^{-1}}$ for a fixed  $\varepsilon>0$
we derive from Lemma~\ref{lem:DoubleSums} 

\begin{cor}
\label{cor:DoubleSums}
For any $\eta>0$ there exists $\delta >0$ 
such that for any  sets
$\cA, \cB \subseteq \F_r$ with $\# \cA \ge r^{1/2 + \eta}$ 
and $\# \cB \ge r^{\eta}$ 
and any non-trivial multiplicative character $\chi$ of $\F_r$,  
we have
$$
\sum_{a\in \cA}\sum_{b \in \cB} \chi(a+b)  \ll \# \cA  \#\cB r^{-\delta},
$$
where the implied constant depends only on $\eta$. 
\end{cor}

We make use of the following special case of 
a result of Shao~\cite{Shao}.

\begin{lemma}
\label{lem:MomentCharSums} Let $\nu\ge 1$ be a fixed integer.  
Let $0 \le w_1 < \ldots  < w_J< p$ 
be $J\ge 1$ arbitrary integers with
$$
w_{j+1} - w_j \ge H, \qquad j = 1, \ldots, J-1,
$$
for some real $H\ge p^{1/2\nu}$.  Then for any non-principal
multiplicative  character $\chi$ modulo $p$, we have 
$$
\sum_{j=1}^{J-1} \max_{h \le H} \left|\sum_{i=1}^{h} \chi\(i+w_j\)\right|^{2\nu} 
\ll p^{1/2 + 1/2\nu + o(1)}H^{2\nu - 2}.
$$
\end{lemma}

\subsection{Bounds of the number of solutions of some 
congruences}
\label{sec:cong}

For $\vec{d} \in \cD_n$ and an integer $q \ge 2$ 
we consider 
the set 
$$
\cN_n(\vec{d}, q) = \left\{s \in \cN_n(\vec{d})~:~ 
 s \equiv 0 \pmod q\right\}. 
$$

\begin{lemma}
\label{lem:Cong Small q} Let $n > k \ge 1$ and
$\vec{d} \in \cD_{k,n}^*$. Then for any 
odd $q$ with  
 $1 < q \le n^{1/10  \kappa}$, where $\kappa = k/n$,  we have
 $$
\# \cN_n(\vec{d}, q) = \frac{1}{q} \#\cN_n(\vec{d}) + 
O\( \#\cN_n(\vec{d}) 2^{-\sqrt{n}}\).
$$
\end{lemma}

\begin{proof} Using the orthogonality of exponential functions
we write
$$
\# \cN_n(\vec{d}, q) = \sum_{s \in \cN_n(\vec{d}) }
\frac{1}{q} \sum_{a =0}^{q-1} \exp(2 \pi i as/q)
= \frac{1}{q} \sum_{a =0}^{q-1}  \sum_{s \in \cN_n(\vec{d}) }
\exp(2 \pi i as/q).
$$
The term corresponding to $a = 0$ is equal to $ \#\cN_n(\vec{d})/q$
while it is easy to see that Lemma~\ref{lem:ExpSums}
also applies to exponential sums with denominators   $q/\gcd(a,q)\ge 3$
instead of $q$.
\end{proof}  
    
For larger values of $q$ we only have an upper bound 
on $\# \cN_n(\vec{d}, q)$. 

For real positive $\kappa$ and $\varrho$ we define 
\begin{equation}
\label{eq:tau}
\tau(\kappa, \varrho) = \frac{1+\varrho - \sqrt{(1-\varrho)^2+4\varrho\kappa}}{2}
\end{equation}
as the root of the equation 
$$
\tau^2 - \tau(1+\varrho) +\varrho(1-\kappa)  = 0
$$
which belongs to the interval $[0, \varrho]$. 
We now set 
\begin{equation}
\label{eq:theta}
\vartheta(\kappa, \varrho)  
=\frac{\tau(\kappa, \varrho)}{\varrho}.
\end{equation}

\begin{lemma}
\label{lem:Cong Med q} Let $\varepsilon > 0$ be 
fixed. Let 
$$n(1-\varepsilon) > k \ge 1 \mand 
2^{n(1-\varepsilon)} \ge q \ge 1.
$$ 
 Then for any $\vec{d} \in \cD_{k,n}^*$  
we have
$$
\# \cN_n(\vec{d}, q) \ll   \#\cN_n(\vec{d}) 
q^{-\vartheta(\kappa, \varrho)}, 
$$
where the implied constant is absolute, 
$\kappa$ and $\varrho$ are defined by
$$
\kappa = k/n \mand q = 2^{\varrho n},
$$
and $\vartheta(\kappa, \varrho)$ is given by~\eqref{eq:theta}.
\end{lemma} 

\begin{proof} We refer to the digits of $s \in  \cN_n(\vec{d})$ 
on positions $j$ with $\delta_j=\free$ as to {\it free positions\/}
and we refer to other digits as to {\it fixed positions\/}. 

We set
$$
r = \rf{\frac{\log q}{\log 2}} -1.
$$
Let $\vec{d} = (\delta_0, \ldots, \delta_{n-1})$. 

We  set $\delta_i=0$ for all integers $i \not \in [0, n-1]$. 

Now, for $j  \in \Z$, we denote by 
 $w_j$  the number of  free positions
amongst the positions $j, \ldots,j+r-1$ and let $\chi^{\free}$ 
be the characteristic function of the symbol `$\free$' defined on the set $\{\free, 0, 1\}$.
Then 
$$
\sum_{j=-r+1}^{n-1} w_j = \sum_{j=-r+1}^{n-1}
\sum_{i=0}^{r-1} \chi^{\free}(\delta_{i+j}) = r(n-k).
$$

We now  set $t = \fl{\tau(\kappa, \varrho)  n}$, where
$\tau(\kappa, \varrho)$ is given by~\eqref{eq:tau}.
We  count the total number $W$ of free positions which appear in 
each of the $n+r -2t+1$ blocks of width $r$  
starting at the points $j = -r+t, \ldots, n-t$.  Then we have 
\begin{equation}
\label{eq:W lower}
W = \sum_{j=-r+t}^{n-t} w_j  =  r(n-k) -
\sum_{j=-r+1}^{-r+t-1} w_j - \sum_{j=n-t+1}^{n-1} w_j .
\end{equation}
We now note that for $j < 0$ we have $w_j \le r-|j|$
and for $j <r$ we have $w_{n-j} \le  j$. 
Hence,
$$
\sum_{j=-r+1}^{-r+t-1} w_j +\sum_{j=n-t+1}^{n-1} w_j 
\le 2\sum_{i=1}^{t-1} i = t^2 + O(t).
$$ 
Therefore, we conclude from~\eqref{eq:W lower} that 
$$
W  \ge r(n-k) - t^2 + O(t).
$$

Hence, for some $h \in [-r+t, n-t]$ we have
\begin{equation}
\label{eq:large w_h}
\begin{split}
w_h & \ge \frac{W}{n+r-2t+1}\ge  \frac{r(n-k) - t^2 + O(t)}{n+r-2t+1}\\
& = 
n\frac{\varrho(1-\kappa) - \tau(\kappa, \varrho)^2}{1+\varrho-2\tau(\kappa, \varrho)} + O(1)\\
& 
= n\tau(\kappa, \varrho) + O(1)
= n \varrho \vartheta(\kappa, \varrho) + O(1), 
\end{split}
\end{equation}
where the implied constant is absolute.

Now fixing the  digits on the remaining $n-k - w_h$ free positions
$j \not \in [h,h+r-1]$
of  the numbers
$$
\sum_{i=0}^{n-1} d_i 2^i \in  \cN_n(\vec{d})
$$ 
and recalling that $2^r < q$, 
we see that the number 
$$
s = \sum_{i=h}^{h+r-1} d_i 2^{i-h}
$$
belongs to a prescribed residue class modulo $q$ and since $0 \le s < 2^r < q$,
$s$ is  uniquely defined. 
Hence, using~\eqref{eq:large w_h}, we obtain
 $$
\# \cN_n(\vec{d}, q) \le 2^{n-k -w_h} 
\le \# \cN_n(\vec{d}, q) 2^{-w_h} \ll
 \# \cN_n(\vec{d}, q) q^{-\vartheta(\kappa, \varrho)},
$$
and the result follows. 
\end{proof}

\begin{lemma}
\label{lem:Cong Aver q} 
Let $\varepsilon>0$ be sufficiently small and
\begin{equation}
\label{eq:rd12}
 \frac{1}{2} \ge \varrho \ge \frac{1}{4}.
\end{equation}
Moreover, let
\begin{equation}
\label{kappa37}
  \kappa=\frac{k}{n}<\frac{3}{7}- 2 \varepsilon
\end{equation}
and
$$
  2^{\varrho n} \ll A \ll 2^{\varrho n},
$$
and suppose that
\begin{equation}
\label{eq:cond1}
(3+4\varepsilon)\varrho  \le  2(1-\kappa) 
\end{equation} 
and
\begin{equation}
\label{eq:cond2}
  \varrho(1+5\varepsilon)<4(1-\kappa)-2.
\end{equation}
Then
\begin{equation}
\label{goal}
  \sum_{A<q \le 2A} \# \cN_n(\mathbf{d}, q^2) \ll
  \# \cN_n(\mathbf{d}) A^{-\varepsilon/2}.
\end{equation}
\end{lemma}

\begin{proof} We follow the definition of free and fixed positions as 
in the proof of Lemma~\ref{lem:Cong Med q}.

Let us divide the set of all $n$ positions into three blocks $W_1$, $W_2$, $W_3$
of consecutive positions 
(from the left to the right) in the following way:
The number of positions in $W_1 \cup W_2$ as well as in $W_2 \cup W_3$
is $2\varrho n+O(1)$. This is certainly possible since we have~\eqref{eq:rd12}
 (more explicitly, $W_1$ and $W_3$ contain 
$n(1-2 \varrho)+O(1)$ positions and $W_2$ contains
$n(4\varrho-1)+O(1)$ positions). 
Let $w_i$ be the number of free positions in block $W_i$, $i=1,2,3$.
Since the total number of free positions is $(1-\kappa)n$, we obtain
\begin{equation}
\label{eq:rd1}
  w_1+w_2+w_3 = (1-\kappa)n.
\end{equation}
Now let $\alpha$ be the number of free positions in $W_1$ and $W_2$
together, that is, $\alpha=w_1+w_2$, and analogously let
$\beta=w_1+w_3$ and $\gamma=w_2+w_3$. Then \eqref{eq:rd1} implies that
\begin{equation}
\label{eq:rd8}
  \alpha+\beta+\gamma = 2(1-\kappa)n.
\end{equation}
Now regarding the neighbouring blocks $W_1$ and $W_2$ as one block
with $\alpha$ free positions, as in the proof of Lemma
\ref{lem:Cong Med q} we obtain
\begin{equation}
\label{eq:N alpha}
\#\cN_n(\mathbf{d}, q^2) \ll \#\cN_n(\mathbf{d}) 2^{-\alpha}
\end{equation}
whenever $A<q \le 2A$. Note that here we  use  the fact that
$A \gg 2^{\varrho n}$, whence $q^2 \gg 2^{2\varrho n}$,
so a congruence modulo $q^2$
fixes all the free positions in the block composed of $W_1$ and $W_2$.
Analogously, we obtain the alternative bound
\begin{equation}
\label{eq:N gamma}
\#\cN_n(\mathbf{d}, q^2) \ll \#\cN_n(\mathbf{d}) 2^{-\gamma},
\end{equation}
using the block composed of $W_2$ and $W_3$. 
Our first observation is that we can assume that 
$\alpha < (1+\varepsilon) \varrho n$, as otherwise~\eqref{eq:N alpha} 
implies 
$$
\#\cN_n(\mathbf{d}, q^2) \ll \#\cN_n(\mathbf{d}) A^{-1-\varepsilon}
$$
and the result follows. 
Similarly, using~\eqref{eq:N gamma}, we can assume that 
 $\gamma < (1+\varepsilon) \varrho n$. Hence, by~\eqref{eq:rd8}
we can also assume that 
\begin{equation}
\label{eq:beta large}
  \beta \ge 2n(1-\kappa-(1+\varepsilon)\varrho).
\end{equation}
Note that by \eqref{eq:cond1}, this implies that
\begin{equation}
\label{betanote}
  2^{-\beta} \ll A^{-1-\varepsilon}.
\end{equation}
Moreover, as trivially $\beta \le (1-\kappa)n$ where
$(1-\kappa)n$ is the total number of free positions, we obtain
$$
  \varrho \ge \frac{1-\kappa}{2(1+\varepsilon)}.
$$
By \eqref{kappa37}, for sufficiently small $\varepsilon>0$ this
implies that
$$
  \varrho>\frac{2}{3}\kappa +\varepsilon,
$$
whence
\begin{equation}
\label{fix}
  \varrho n - \beta + (2-6\varrho)n \le
  n (2\kappa-3\varrho+2\varepsilon \varrho) \le -\varepsilon n
\end{equation}
for sufficiently small $\varepsilon>0$.

Working with $W_1$ and $W_3$
is more difficult, as we are no longer dealing with one, but rather with
two intervals. Writing $r$ for the bit position at the right of $W_1$,
and $s$ for the position at the right of $W_2$,
we are now considering congruences of the form
\begin{equation}
\label{eq:rd2}
  2^ra+2^s b+c \equiv 0 \pmod {q^2}.
\end{equation}
Note that from the construction of $W_1$, $W_2$, $W_3$ it follows that
$$
  r \ge 2\varrho n+O(1).
$$
Once $b$, corresponding to $W_2$, has been fixed, the solution set of
\eqref{eq:rd2} is of the form 
\begin{equation}
\label{eq:ac}
(a,c) = (a_0, c_0) + (a_1, c_1), 
\end{equation}
where $(a_0, c_0) \in \Z^2$ is a 
fixed solution of~\eqref{eq:rd2}
and $(a_1, c_1)$
runs over all solutions of the homogeneous congruence
\begin{equation}
\label{eq:rd3}
  2^ra_1 + c_1 \equiv 0 \pmod {q^2}.
\end{equation}
By construction of the $W_i$, we see that $a$ and $c$ are non-negative integers
with $a, c \ll 2^{(1-2\varrho)n}$, whence also
$|a_1|, |c_1| \ll 2^{(1-2\varrho)n}$.
Moreover, the congruence~\eqref{eq:rd3} describes a two-dimensional lattice
with a basis $\{(1, -2^r), (0, q^2)\}$ and
of determinant $q^2$. Let $2^{\lambda_1(q)}, 2^{\lambda_2(q)}$ be its
successive minima, where $\lambda_1(q) \le \lambda_2(q)$.
For the general background on lattices 
we refer to~\cite{GrLoSch}. 

Then
$$
  q^2 \ll 2^{\lambda_1(q)+\lambda_2(q)} \ll q^2.
$$
Let us first discuss the case that
$$
  \lambda_2(q) \le (1-2\varrho)n.
$$
Then the number of solutions to~\eqref{eq:rd3} with 
$|a_1|, |c_1| \ll 2^{(1-2\varrho)n}$ can be estimated as
$$
O\( \(2^{(1-2\varrho)n-  \lambda_1(q)}+1\) \(2^{(1-2\varrho)n-  \lambda_2(q)}+1\) \)
= O\(2^{2(1-2\varrho)n}q^{-2}\)
$$ 
(note that 
$q^2 \ll 2^{\lambda_1(q) + \lambda_2(q)}
\le 2^{2 \lambda_2(q)} \le 2^{2(1-2\varrho)n}$). 
Furthermore, since  $q^2\ge A^2  \gg 2^{2 \varrho n}$, we obtain the 
bound 
$O\(2^{(2-6\varrho)n}\)$ for the number of solutions to~\eqref{eq:rd3}.

Considering all the possible 
$$
2^{w_2} = 2^{n- k-\beta} = \# \cN_n(\mathbf{d}) 2^{-\beta}
$$
choices for $b$,
we therefore obtain
$$
\# \cN_n(\mathbf{d}, q^2) \ll \# \cN_n(\mathbf{d}) 2^{(2-6\varrho)n-\beta}.
$$
By~\eqref{fix}, this contribution,
when summed over $A<q \le 2A$, is negligible with respect 
to~\eqref{goal}. 

We may therefore without loss of generality assume that
$$
  \lambda_2(q)>(1-2\varrho)n.
$$
Again,
from~\eqref{eq:ac} we conclude that  the number of solutions
of~\eqref{eq:rd2} with $|a_1|, |c_1| \ll 2^{(1-2\varrho)n}$ is
$$
O\( \(2^{(1-2\varrho)n-  \lambda_1(q)}+1\) \(2^{(1-2\varrho)n-  \lambda_2(q)}+1\) \) 
= O\(2^{(1-2\varrho)n-\lambda_1(q)} + 1\),
$$ and the number of possible
choices for $b$ is bounded by $\# \cN_n(\mathbf{d}) 2^{-\beta}$,
so
\begin{equation}
\label{eq:N beta}
\# \cN_n(\mathbf{d}, q^2) \ll
\# \cN_n(\mathbf{d}) 2^{-\beta} (2^{(1-2\varrho)n-\lambda_1(q)}+1).
\end{equation}
Let us define a real parameter 
\begin{equation}
\label{eq:lambda}
\lambda = (1-2\varrho)n-2(1-\kappa)n+3(1+\varepsilon)\varrho n. 
\end{equation}
If $\lambda_1(q)>  \lambda$,
then~\eqref{eq:beta large} and~\eqref{betanote} give
$$
  2^{-\beta}\(2^{(1-2\varrho)n-\lambda_1(q)} +1\)
\le   2^{-(1+\varepsilon)\varrho n} + 2^{-\beta} \ll A^{-1-\varepsilon},
$$
so
$$
  \sum_{\substack{A<q \le 2A:\\ \lambda_1(q) \ge \lambda}}
  \#\cN_n(\mathbf{d}, q^2) \ll \# \cN_n(\mathbf{d}) A^{-\varepsilon}.
$$

It now remains to estimate the contribution from $q$ with $ \lambda_1(q) \le \lambda$. 
Furthermore, it is enough to show that for any real positive $\mu< \lambda$ we have 
\begin{equation}
\label{eq:rd13}
  \sum_{\substack{A<q \le 2A:\\ \mu \le \lambda_1(q) < \mu+1}}
  \#\cN_n(\mathbf{d}, q^2) \ll \# \cN_n(\mathbf{d}) A^{-\varepsilon}.
\end{equation}
Now $\lambda_1(q) \le \mu+1$ means that there exist $a_1, c_1 \in \Z$,
not both zero, such that $|a_1|, |c_1| \ll 2^\mu$ 
and~\eqref{eq:rd3} holds true. Note that $2^r a_1 + c_1 = 0$ is impossible,
as it  implies that $2^r \mid c_1$, so $|c_1| \ge 2^r \gg
2^{2\varrho n}$,
contradicting $|c_1| \ll 2^\mu$ as by~\eqref{eq:rd12} and~\eqref{eq:cond1}
we have
$$\mu<\lambda \le  (1-2\varrho)n- \varepsilon\varrho n 
\le 2\varrho n- \varepsilon\varrho n.
$$

Therefore, $2^r a_1 + b_1 \ne 0$,
so by~\eqref{eq:rd3} for each fixed pair $(a_1, c_1)$ there are only
$2^{o(n)}$ possibilities for $q$ that are integer divisors of 
$2^r a_1 + b_1 = O(2^n)$, see~\cite[Theorem~317]{HardyWright}. 
The number of possible
$(a_1, c_1)$ can be bounded by $O(2^{2\mu})$, and
$\# N_n(\mathbf{d}, q^2)$, by~\eqref{eq:N beta}, is at most of order of
magnitude
$$
\#\cN_n(\mathbf{d})(2^{-\beta+(1-2\varrho)n-\mu}+2^{-\beta}) \ll
\#\cN_n(\mathbf{d})(2^{-\beta+(1-2\varrho)n-\mu}+A^{-1-\varepsilon}).
$$
We therefore obtain
\begin{align*}
  \sum_{\substack{A<q \le 2A:\\ \mu \le \lambda_1(q) < \mu+1}}
  \#\cN_n(\mathbf{d}, q^2) & \ll \# \cN_n(\mathbf{d})
  (2^{-\beta+(1-2\varrho)n+\mu+n\varepsilon} + A^{-\varepsilon}) \\
  & \ll \#\cN_n(\mathbf{d})
        (2^{-\beta+(1-2\varrho)n+\lambda+o(n)}+A^{-\varepsilon}).
\end{align*}
Now by
\eqref{eq:cond2}, \eqref{eq:beta large} and~\eqref{eq:lambda}, we have
$$
  -\beta+(1-2\varrho)n+ \lambda < 0,
$$
completing the proof.
\end{proof}

In particular, covering an interval $[A,B]$ by dyadic 
intervals, we derive from Lemma~\ref{lem:Cong Aver q}  the following result suitable for 
our applications. 

\begin{cor}
\label{cor:Cong Aver q} 
Let $\varepsilon>0$ and 
$$
  \kappa=\frac{k}{n}<\frac{3}{7}-2\varepsilon
$$
Suppose that 
$$
\frac{1}{2} \ge \zeta \ge  \xi \ge \frac{1}{4}, 
\qquad (3+4\varepsilon)\zeta \le   2(1-\kappa), \qquad 
  \zeta(1+5\varepsilon) < 4(1-\kappa)-2.
$$
Then for any $A$ and $B$ with 
$$
 2^{\xi n} \ll A \le B  \ll 2^{\zeta n}
$$
we have 
$$
  \sum_{A<q \le B} \# \cN_n(\mathbf{d}, q^2) \ll
  \# \cN_n(\mathbf{d}) A^{-\varepsilon/2} \log B.
$$
\end{cor}

\begin{lemma}
\label{two_windows}
Keeping the notation of Lemma \ref{lem:Cong Med q}, suppose that
$$
  \frac{\kappa}{2} \le \varrho \le \frac{1}{2}.
$$
Then
$$
\# \cN_n(\vec{d}, q) \ll   \#\cN_n(\vec{d}) 
q^{-1+\kappa/2\varrho}.
$$
\end{lemma}
\begin{proof}
We use a similar, but simpler argument as in the proof of Lemma
\ref{lem:Cong Aver q}. As $\varrho \le \frac{1}{2}$, we can divide the
$n$ bits into three blocks $W_1, W_2, W_3$ (from the left to the right),
such that $W_1$ and $W_3$ have size $\varrho n + O(1)$, each, and $W_2$
has size $(1-2\varrho)n+O(1)$. Then in one of $W_1$ and $W_3$, say $W_1$,
there must be at least
$$
  \frac{1}{2}\left((1-\kappa)n-(1-2\varrho)n\right) + O(1) =
  n(\varrho-\kappa/2)+O(1)
$$
many free positions. Once all the bits outside $W_1$ have been chosen,
a congruence modulo $q$ fixes all $\gg n(\varrho-\kappa/2)$ remaining
free positions in $W_1$, whence
$$
\# \cN_n(\vec{d}, q) \ll \#\cN_n(\vec{d}) 2^{-n(\varrho-\kappa/2)}
\ll \#\cN_n(\vec{d}) q^{-1+\kappa/2\varrho},
$$
which concludes the proof. 
\end{proof}

To apply Lemma~\ref{lem:Cong Med q} we also need the 
following technical statement.

\begin{lemma}
\label{lem:theta} For  $1> \kappa> 0$
the function $\vartheta(\kappa, \varrho)$ given by~\eqref{eq:theta}
is monotonically decreasing as $\varrho$ is increasing. 
\end{lemma} 

\begin{proof} 
The result follows from the observation that the derivative
$$
\frac{{\partial } \vartheta}{{\partial } \varrho}=
\frac{1 + (-1 + 2 \kappa) \varrho - \sqrt{1 + (-2 + 4 \kappa) \varrho + \varrho^2}}{
2 \varrho^2 \sqrt{ 1 + (-2 + 4 \kappa) \varrho + \varrho^2}}
$$
is negative, as $1+2c\varrho+ \varrho^2>0$ and $1+c\varrho -\sqrt{1+2c\varrho+ \varrho^2}<0$ when $|c|<1$. 
\end{proof}

\subsection{Some results from additive combinatorics}
\label{sec:addcomb}

We now recall a recent result by Schoen~\cite[Theorem~3.3]{Sch}
in additive combinatorics. As in~\cite{DES1}, we 
note that~\cite[Theorem~3.3]{Sch} is only stated
for subset sums but 
can be easily extended to  Hilbert cubes. 

\begin{lemma}
\label{lem:hilbcube}
For any $a_0 \in \F_p$ and pairwise distinct $a_1, \ldots, a_d \in \F_p$
such that $d \ge 8(p/\log p)^{1/D}$, where $D$ is an integer satisfying
$$
0<D\le \sqrt{\frac{\log p}{2 \log \log p}},
$$  
the Hilbert cube~\eqref{eq:hilbert_cube} contains an arithmetic progression
of length $L$ where
$$
  L \ge 2^{-10} (d/\log p)^{1+1/(D-1)}.
$$
\end{lemma}

For a set 
$\cS \subseteq \F_p$ 
we use $\Sigma_k(\cS)$ to denote the set 
of all $k$-elements subset sums of $\cS$, that is,
$$
\Sigma_k(\cS) = \left\{\sum_{t \in \cT} t~:~
\cT \subseteq \cS, \ \#\cT=k\right\}.
$$
We make use of the following result of  Dias da Silva and
Hamidoune~\cite[Theorem~4.1]{DiDaSiHa}.

\begin{lemma}
\label{lem:subset}
For a set $\cS \subseteq \F_p$ and an integer $k \ge 1$, we have 
$$
\# \Sigma_k(\cS) \ge \min\{p, k\#\cS -k^2 +1\}.
$$
\end{lemma}

We now define
$$
\Sigma_*(\cS) = \bigcup_{k=0}^{\# \cS} \Sigma_k(\cS).
$$

Taking $k = \fl{\#\cS/2}$ in Lemma~\ref{lem:subset} we immediately 
derive:

\begin{cor}
\label{cor:subset}
For a set $\cS \subseteq \F_p$ and an integer $k \ge 1$ we have 
$$
\# \Sigma_*(\cS) \gg \min\{p, (\#\cS)^2 +1\}.
$$
\end{cor}

\section{Main Results}

\subsection{Squarefree integers with fixed digits}
\label{sec:sqfr}
Let $S_n(\vec{d})$ be the number of squarefree integers 
$s \in \cN_n(\vec{d})$. 

\begin{theorem}
\label{thm:SF} 
For any $\varepsilon > 0$, uniformly 
over integer  $k <(2/5 - \varepsilon) n$ and 
$\vec{d} \in \cD_{k,n}^*$, we have
$$
S_n(\vec{d})= \(\frac{8}{\pi^2} +o(1)\) \#\cN_n(\vec{d}).
$$ 
\end{theorem}

\begin{proof} The inclusion-exclusion principle yields
$$
S_n(\vec{d}) = \sum_{q=1}^\infty \mu(q) \# \cN_n(\vec{d}, q^2), 
$$
where $\mu(q)$ is the M{\"o}bius function, see~\cite[Section~16.3]{HardyWright}.

As before, we define $\kappa = k/n$ and we also use the
function $\vartheta(\kappa, \varrho)$ that is given by~\eqref{eq:theta}.

For $\kappa<2/5-\varepsilon$ we have
\begin{equation}
\label{eq:25}
  \vartheta(\kappa, 2/5)>1/2
\end{equation}
as $\vartheta(2/5, 2/5)=1/2$ and for fixed $\varrho$, the function
$\vartheta(\kappa, \varrho)$ given by \eqref{eq:theta} is obviously decreasing in
$\kappa$.
Choose $\zeta>2/5$ such that
\begin{equation}
\label{eq:zeta}
 (3+4\varepsilon) \zeta \le  2(1-\kappa), \qquad 
   \zeta (1+5 \varepsilon) < 4(1-\kappa)-2,
\end{equation}
which for sufficiently small $\varepsilon>0$ is possible since
$\kappa< 2/5 - \varepsilon$. 
Note that in particular, $\zeta>\kappa$.

We set 
$$
T = n^{1/20  \kappa}, \qquad U=2^{n/5},
\qquad V = 2^{n/4},  \qquad W = 2^{\zeta n}, 
$$
and  write
\begin{equation}
\label{eq:Si}
S_n(\vec{d}) = S_1 + S_2 + S_3 + S_4 + S_5,
\end{equation}
where
\begin{equation*}
\begin{split}
S_1 & =\sum_{q \le T} \mu(q) \# \cN_n(\vec{d}, q^2),\\
S_2 & =\sum_{T< q \le U} \mu(q) \# \cN_n(\vec{d}, q^2),\\
S_3 & =\sum_{U< q \le V} \mu(q) \# \cN_n(\vec{d}, q^2),\\
S_4 & =\sum_{V < q \le  W} \mu(q) \# \cN_n(\vec{d}, q^2),\\
S_5 & =\sum_{q > W} \mu(q) \# \cN_n(\vec{d}, q^2).
\end{split}
\end{equation*}
We use Lemma~\ref{lem:Cong Small q}  for $q \le T$, 
getting the main term
\begin{equation*}
\begin{split}
S_1 &=  \#\cN_n(\vec{d}) \sum_{\substack{q \le T\\ q~\text{odd}}} \frac{\mu(q)}{q^2}  + 
O\( \#\cN_n(\vec{d}) T 2^{-\sqrt{n}}\) \\
& =   \#\cN_n(\vec{d}) \sum_{q~\text{odd}} \frac{\mu(q)}{q^2}  + 
O\( \#\cN_n(\vec{d}) T 2^{-\sqrt{n}} + \#\cN_n(\vec{d}) T^{-1} \) \\
& =   \#\cN_n(\vec{d}) \prod_{\substack{\ell \ge 3\\
\ell~\text{prime}}}
\(1 - \frac{1}{\ell^2}\)  + 
O\( \#\cN_n(\vec{d}) T 2^{-\sqrt{n}} + \#\cN_n(\vec{d}) T^{-1}\)\\
& =   \frac{4}{3}\#\cN_n(\vec{d}) \prod_{\ell~\text{prime}}
\(1 - \frac{1}{\ell^2}\)  + 
O\( \#\cN_n(\vec{d}) T 2^{-\sqrt{n}} + \#\cN_n(\vec{d}) T^{-1}\).
\end{split}
\end{equation*}
So we now obtain the main term
\begin{equation}
\label{eq:S1}
S_1 =   \(\frac{8}{\pi^2} + o(1)\)\#\cN_n(\vec{d}), 
\end{equation}
see~\cite[Theorem~280]{HardyWright}. 

To estimate $S_2$, we use  Lemma~\ref{lem:Cong Med q}. 
First we note that by Lemma~\ref{lem:theta}, for  $T < q \le U$, 
we have 
$$
\vartheta(\kappa, 2\varrho)  \ge  \vartheta(\kappa, 2/5), 
$$
where, in analogy to Lemma~\ref{lem:Cong Med q},  $\varrho$ is defined by
$q = 2^{\varrho n}$. Hence  in this range we have
$$
\# \cN_n(\vec{d}, q^2) \ll
 \#\cN_n(\vec{d}) q^{-2\vartheta(\kappa, 2\varrho)} \ll \#\cN_n(\vec{d}) q^{-2\vartheta(\kappa, 2/5)}.
$$
Since by~\eqref{eq:25} we have $2\vartheta(\kappa, 2/5)> 1$,  we now derive
\begin{equation}
\label{eq:S2}
S_2 \ll \#\cN_n(\vec{d}) T^{1-2\vartheta(\kappa, 2/5)} = o\(\#\cN_n(\vec{d})\). 
\end{equation}

For $S_3$ we use a similar argument as for $S_2$, 
this time with Lemma~\ref{two_windows}
instead of Lemma~\ref{lem:Cong Med q}, noting that with $\kappa<2/5$
and $\varrho \ge 1/5$ we obtain
$$
\# \cN_n(\vec{d}, q^2) \ll \#\cN_n(\vec{d}) q^{2(-1+\kappa/4\varrho)}
\ll \#\cN_n(\vec{d}) q^{-1-\delta}
$$
for some sufficiently small $\delta>0$, so again
\begin{equation}
\label{bound_s3}
  S_3 = o\(\#\cN_n(\vec{d})\). 
\end{equation}

To estimate $S_4$, we use Corollary~\ref{cor:Cong Aver q}, which
by~\eqref{eq:zeta} 
applies with some sufficiently small $\varepsilon >0$,  getting 
\begin{equation}
\label{eq:S3}
S_4 \ll \# \cN_n(\mathbf{d}) V^{-\varepsilon/2} \log W = o\(\#\cN_n(\vec{d})\).
\end{equation}

Finally, we use the trivial bound $\# \cN_n(\vec{d}, q^2) \le 2^n/q^2$
for $q > W$ and using $\zeta>\kappa$ we derive 
\begin{equation}
\label{eq:S4}
S_5 \ll 2^n W^{-1} \ll \#\cN_n(\vec{d}) 2^{\kappa n - \zeta n} 
= o\(\#\cN_n(\vec{d})\). 
\end{equation}
Substituting~\eqref{eq:S1}, \eqref{eq:S2}, \eqref{bound_s3}, \eqref{eq:S3} and~\eqref{eq:S4} into~\eqref{eq:Si}, we now conclude the 
proof.  
 \end{proof}

 \subsection{Average values of the Euler function}
 \label{sec:Euler}

We now consider the average value
$$
F_n(\vec{d}) = \sum_{s\in \cN_n(\vec{d})} \frac{\varphi(s)}{s}
$$
with the Euler function $\varphi(s)$, see~\cite[Section~16.3]{HardyWright}.

\begin{theorem}
\label{thm:AverEuler} 
For any $\varepsilon > 0$, uniformly 
over integers  $k <(1 - \varepsilon) n$ and 
$\vec{d} \in \cD_{k,n}^*$, we have
$$
F_n(\vec{d})= \(\frac{8}{\pi^2} +o(1)\) \#\cN_n(\vec{d}).
$$ 
\end{theorem}

\begin{proof} Using the the well-known formula
$$
 \frac{\varphi(s)}{s} = \sum_{q \mid s}  \frac{\mu(q)}{q}
$$
see~\cite[Equation~(16.3.1)]{HardyWright}, and changing the order 
of summation, we write
$$
F_n(\vec{d}) = \sum_{q=1}^\infty \frac{\mu(q)}{q}  \# \cN_n(\vec{d}, q).
$$

We now proceed very similarly to the proof of Theorem~\ref{thm:SF}.

Again we define $\kappa = k/n$ and we also use the
function $\vartheta(\kappa, \varrho)$ that is given by~\eqref{eq:theta}.

Clearly, for $0 < \kappa < 1$ we have 
$$
\vartheta(\kappa, 1) =  1 - \sqrt{\kappa} > 0.
$$
Thus, using Lemma~\ref{lem:theta}, we see that for any $\kappa<  1$
we can find $\xi$ to satisfy 
\begin{equation}
\label{eq:xi11}
1> \xi > \kappa 
\end{equation}
and
\begin{equation}
\label{eq:xi22}
\vartheta(\kappa, \xi) > 0.
\end{equation}

We set 
$$
Q = n^{1/10  \kappa} \mand W = 2^{\xi n}.
$$
and  write
\begin{equation}
\label{eq:Ti}
F_n(\vec{d}) = T_1 + T_2 + T_3,
\end{equation}
where
\begin{equation*}
\begin{split}
T_1 & =\sum_{q \le Q}\frac{\mu(q)}{q}  \# \cN_n(\vec{d}, q),\\
T_2 & =\sum_{Q< q \le W}\frac{\mu(q)}{q}  \# \cN_n(\vec{d}, q),\\
T_3 & =\sum_{q > W} \frac{\mu(q)}{q}  \# \cN_n(\vec{d}, q).
\end{split}
\end{equation*}
We use Lemma~\ref{lem:Cong Small q}  for $q \le Q$, and exactly as in the 
proof of of Theorem~\ref{thm:SF} we obtain the main term
\begin{equation}
\label{eq:T1}
T_1 =     \#\cN_n(\vec{d}) \sum_{\substack{q \le Q\\ q~\text{odd}}} \frac{\mu(q)}{q^2}  + 
O\( \#\cN_n(\vec{d}) Q 2^{-\sqrt{n}}\) =\(\frac{8}{\pi^2} + o(1)\)\#\cN_n(\vec{d}),
\end{equation}
see~\cite[Theorem~280]{HardyWright}. 

To estimate $T_2$, we use  Lemma~\ref{lem:Cong Med q}  for $Q < q \le W$. 
First we note that by Lemma~\ref{lem:theta}, for $Q < q \le W$, 
we have 
$$
\vartheta(\kappa, \varrho)  \ge  \vartheta(\kappa, \xi), 
$$
where, in analogy to Lemma~\ref{lem:Cong Med q},  $\varrho$ is defined by
$q = 2^{\varrho n}$. Hence  in this range we have
$$
\# \cN_n(\vec{d}, q) \ll
 \#\cN_n(\vec{d}) q^{-\vartheta(\kappa, \varrho)} \ll \#\cN_n(\vec{d}) q^{-\vartheta(\kappa, \xi)}.
$$
Since by~\eqref{eq:xi22} we have $\vartheta(\kappa, \xi)> 0$,  we now derive
\begin{equation}
\label{eq:T2}
T_2 \ll \#\cN_n(\vec{d}) Q^{-\vartheta(\kappa, \xi)} = o\(\#\cN_n(\vec{d})\). 
\end{equation}

Finally, we use the trivial bound $\# \cN_n(\vec{d}, q) \le 2^n/q$
for $q > W$ and using~\eqref{eq:xi11} derive 
\begin{equation}
\label{eq:T3}
T_3 \ll 2^n W^{-1} \ll \#\cN_n(\vec{d}) 2^{\kappa n - \xi n} 
= o\(\#\cN_n(\vec{d})\). 
\end{equation}
Substituting~\eqref{eq:T1}, \eqref{eq:T2} and~\eqref{eq:T3} into~\eqref{eq:Ti}, we conclude the 
proof.  
 \end{proof}

\subsection{Non-residues with fixed digits}
\label{sec:nonres}
For a prime $p$, we use $\cN_n^{+}(\vec{d}, p)$ and
$\cN_n^{-}(\vec{d}, p)$ to denote the sets of $s \in \cN_n(\vec{d})$,
that are quadratic residues and non-residues, respectively
(we also use $ \cN_n^{\pm}(\vec{d}, p)$ to denote either of these sets).

\begin{theorem}
\label{thm:NonRes} For any $\varepsilon > 0$ there exists
some $\delta > 0$ such that for $k < (1/2 - \varepsilon) n$, 
$\vec{d} \in \cD_{k,n}$  and 
any prime $p$ with 
$2^{n}< p < 2^{n+1}$ 
we have
$$
\# \cN_n^{\pm}(\vec{d}, p) = \(\frac{1}{2} + O(p^{-\delta})\) \#\cN_n(\vec{d}).
$$ 
\end{theorem} 

\begin{proof} We select arbitrary $s= \rf{\varepsilon n/2}$ 
free positions and denote by $\cB$ the set of $2^s$ integers
with all possible combinations of digits on these positions 
and zeros on all other positions. We also define by $\cA$ 
the subset of $2^{n-k - s}$ elements 
of $\cN_n(\vec{d})$
which also have zero digits on the positions that are 
allocated to $\cB$. 
Clearly each element of $\cN_n(\vec{d})$
has a unique representation as $a+b$ with $a \in \cA$, $b \in \cB$. 
The result is now instant from Corollary~\ref{cor:DoubleSums}.
\end{proof}

\subsection{Primitive roots in Hilbert cubes}
\label{sec:Hilb}
We now present an improvement of~\cite[Theorem~1.3]{DES1}.

\begin{theorem}
\label{thm:Hilb}
We have 
$$
F(p) \le p^{3/19+o(1)}.
$$
\end{theorem}

\begin{proof} 
Let ${\cH}(a_0;a_1, \ldots, a_d)$ be a Hilbert cube, with $d$ 
distinct base
elements $a_0, \ldots, a_d \in \F_p$.
Suppose that
$ {\cH}(a_0;a_1, \ldots, a_d)$ does not contain primitive roots modulo $p$. 
We show that  $d=O(p^{3/19+o(1)})$.

As in the proof of~\cite[Theorem~1.3]{DES1}, we fix some 
$\varepsilon > 0$ and assume that 
\begin{equation}
\label{eq:d}
d \ge p^{3/19 + \varepsilon},
\end{equation}
and without loss of generality we may also assume that $d$ is even.

Let ${\cU} = {\cH}(a_0;a_1, \ldots, a_{d/2})$ and
 ${\cV} = \cH(0;a_{d/2+1}, \ldots, a_{d})$, both $\cU$ and $\cV$ understood
as subsets of $\F_p$. 
It follows from  Lemma~\ref{lem:hilbcube}, applied with $D = 7$ 
that $\cU$ contains
an arithmetic progression $\cA \subseteq  \F_p$ of length 
\begin{equation}
\label{eq:L}
L = \# \cA \geq p^{7/38+o(1)}.
\end{equation}
Let $\Delta$ be the difference between 
consecutive terms of the progression.
Let us consider the interval 
$$
\cI=\{\Delta^{-1}a~:~a\in \cA\} \subseteq   \F_p
$$
of $L$ consecutive residues modulo $p$.  

On the other hand, it follows from Corollary~\ref{cor:subset} 
that $\# \cV \gg d^2$.
Now let 
$$
\cW = \{\Delta^{-1}a~:~a\in \cV\} \subseteq   \F_p.
$$
We now take an arbitrary element $w_1 \in \cW$ and remove from 
$\cW$ at most $O(L)$ elements $w$ with $|w-w_1| \le L$ and 
denote the remaining set as $\cW_1$.
We now choose an arbitrary element $w_2 \in \cW_1$  
and remove from 
$\cW_1$ at most $O(L)$ elements $v$ with $|w-w_2| \le L$ and 
denote the remaining set as $\cW_2$. Continuing, we obtain a set 
$\{w_1, \ldots, w_J\}$ of 
\begin{equation}
\label{eq:J}
J \gg \# \cW/L = \# \cV/L \gg d^2/L
\end{equation}
elements, which after renumbering 
satisfy the condition of Lemma~\ref{lem:MomentCharSums}.

By Lemma~\ref{lem:MomentCharSums} and \eqref{eq:J}, taking a sufficiently 
large $\nu$ after simple calculations we obtain that for any 
non-trivial multiplicative character  $\chi$ of $\F_p$ we have 
$$
 \sum_{j=1}^{J-1}  \left|\sum_{i=1}^{L} \chi\(i+w_j\)\right|^{2\nu} 
\ll p^{1/2 + 1/2\nu + o(1)}L^{2\nu - 2} \ll JL^{2\nu} p^{-\eta},
$$
provided that
$$
  L^{1/2} d \gg p^{1/4+\varepsilon}.
$$
Recalling \eqref{eq:d} and \eqref{eq:L}, we find that the latter condition
is satisfied.
Expressing, in a standard fashion, 
the counting function for primitive roots among the elements 
$i+w_j$, $i =1, \ldots, L$, $j = 1, \ldots, J$, via multiplicative characters,
see, for example,~\cite[Lemma~2.4]{DES1}, we see that it is positive, 
provided $p$ is large enough. 
Since $\varepsilon$ is arbitrary, 
we obtain the desired result. 
 \end{proof}

\subsection{Elements with restricted digits in finite fields}
\label{sec:rest dig}

Our next result improves~\cite[Theorem~2.1]{DMS} for any  $n \ge 3$:

For $\cW_\fA$,
given by~\eqref{eq:Set WA}, 
we use $\cW_\fA^{+}$ and
$\cW_\fA^{-}$ to denote the sets of $w \in \cW_\fA$,
that are quadratic residues and non-residues, respectively
(we also use $\cW_\fA^{\pm}$ to denote either of these sets).

\begin{theorem}
\label{thm:Nonres W}
Let $n \ge 2$. 
For any $\varepsilon > 0$ there is some $\delta > 0$ 
such that  for an  arbitrary
basis $\omega_1, \ldots, \omega_n$ of $\F_{p^n}$ 
over $\F_p$ and a collection of $n$ sets 
$$
\fA = \{\cA_i\subseteq \F_p~:~i=1, \ldots, n\},
$$
satisfying
\begin{equation}
\label{eq:DES Cond1}
 \prod_{1 \le i \le n}\# \cA_i  \ge  p^{(1/2 + \varepsilon)n^2/(n-1)}
\end{equation}
and
\begin{equation}
\label{eq:DES Cond2}
\min_{1 \le i \le n}\# \cA_i  \ge  p^\varepsilon 
\end{equation}
the following holds: for the set
$$
\cW_\fA = \left\{a_1 \omega_1+\ldots+a_n \omega_n~:~ 
a_i \in \cA_i, \ i=1, \ldots, n\right\},
$$
we have
$$
\# \cW_\fA^{\pm} = \(\frac{1}{2} + O(p^{-\delta})\) \# \cW_\fA
$$
uniformly over $n$ and $p$. 
\end{theorem}

\begin{proof}Without loss of generality, we can assume that 
$\varepsilon \le 1/2$. 
We also set
$$
n_0(\varepsilon) = \rf{4\varepsilon^{-1}}. 
$$

We first  
consider the case when 
\begin{equation}
\label{eq:large n}
n >  n_0(\varepsilon).
\end{equation}

Assuming that~\eqref{eq:large n} holds, we set
$$
m = \rf{\frac{1 + \varepsilon}{1+2\varepsilon} n}.
$$
By choosing $\cI$ such that
 the sets $\cA_i$ with $i \in \cI$ are the $m$ sets of largest cardinality
we see that
\begin{equation}
\label{eq:large Ai}
 \prod_{i\in\cI }\# \cA_i  \ge  
 \(\prod_{1 \le i \le n}\# \cA_i\)^{m/n} . 
\end{equation}
We then  define $\cJ = \{1, \ldots, n\} \setminus \cI$ and 
\begin{equation*}
\begin{split}
\cA & = \left\{\sum_{i\in \cI} a_i \omega_i~:~ 
a_i \in \cA_i, \ i\in \cI\right\}, \\
\cB & = \left\{\sum_{i \in \cJ} a_i \omega_i~:~ 
a_i \in \cA_i, \ i\in \cJ\right\}.
\end{split}
\end{equation*}
Thus, recalling~\eqref{eq:DES Cond1} and~\eqref{eq:large Ai}, we see that 
\begin{equation}
\label{eq:Card A}
\begin{split}
\# \cA  \ge   \(p^{(1/2 + \varepsilon)n^2/(n-1)}\)^{m/n}
\ge     \(p^{(1/2 + \varepsilon)n}\)^{m/n}\ge
 p^{(1/2 + \varepsilon)m} \ge  p^{(1/2 + \varepsilon/2)n} .
\end{split}
\end{equation}
Furthermore for $n > n_0(\varepsilon)$ we have
\begin{equation*}
\begin{split}
\#\cJ & = n-m   \ge n - \frac{1 + \varepsilon}{1+2\varepsilon} n -1
= n - \frac{(1 + \varepsilon)n + 1+2\varepsilon}{1+2\varepsilon} \\
&\ge  n - \frac{(1 + \varepsilon)n + 2}{1+2\varepsilon}  
=  \frac{ \varepsilon n - 2}{1+2\varepsilon}  \ge 
 \frac{ \varepsilon n - 2}{2} \ge \varepsilon n/4.
\end{split}
\end{equation*}
Therefore, recalling~\eqref{eq:DES Cond2} again, we see that
\begin{equation}
\label{eq:Card B}
\begin{split}
\# \cB \ge p^{\varepsilon^2 n/4}.
\end{split}
\end{equation}
As $\{\omega_1, \ldots, \omega_n\}$ is a basis,
every element $w = \cW_\fA$ has a unique
representation $w =a + b$ with $a \in \cA$, $b \in \cB$.
Hence for any multiplicative character $\chi$ of $\F_r$ we have 
$$
\sum_{w\in \cW_\fA} \chi(w)  
= \sum_{a\in \cA}\sum_{b \in \cB} \chi(a+b) .
$$
Now using the bounds~\eqref{eq:Card A} 
and applying~\eqref{eq:Card B} and  applying Corollary~\ref{cor:DoubleSums}
with $\eta = \varepsilon^2/4$, we obtain 
\begin{equation}
\label{eq:Bound W_A}
\sum_{w\in \cW_\fA} \chi(w)   \ll \#\cA \# \cB   p^{-\delta}
=  \#\cW_\fA    p^{-\delta}
\end{equation}
for any non-trivial multiplicative character  $\chi$ 
of $\F_{p^n}$ where $\delta > 0$ depends only on $\varepsilon$. 
Thus, taking the quadratic character $\chi$, we obtain the desired result
if the inequality~\eqref{eq:large n} holds.

Now, for small values of $n$, for which inequality~\eqref{eq:large n} fails,
we simply take $m = n-1$ and choose $\cI$ as before, 
to satisfy~\eqref{eq:large Ai}. 
Hence, instead of~\eqref{eq:Card A} we have
$$
\# \cA  \ge  p^{(1/2 + \varepsilon)n} 
$$
and also trivially, we have 
$$
\#\cB \ge  p^{ \varepsilon} \ge  p^{\varepsilon n/n_0(\varepsilon)}.
$$
Applying Corollary~\ref{cor:DoubleSums}
with $\eta =\varepsilon/n_0(\varepsilon)$, we obtain the bound~\eqref{eq:Bound W_A}
again, which concludes the proof. 
\end{proof}
 
\section{Comments}

We remark that motivated by a question of Erd{\H o}s, Mauduit and
S{\'a}rk{\"o}zy~\cite[Problem~5]{EMS}, Banks and Shparlinski~\cite{BaSh1}
have studied 
the average value of $\varphi(s)/s$ over
integers with some digital restrictions, different from that of Theorem~\ref{thm:AverEuler},
see also~\cite{BaSh2}. 
As in~\cite{BaSh2}, one can also study the average value of $\sigma(s)/s$ for the sums
of divisors function and obtain a full analogue of Theorem~\ref{thm:AverEuler}
for this function. 
Furthermore, our argument can be used to give an asymptotic formula 
for the number of  pairs $(s,r)$ with $r,s \in \cN_n(\vec{d})$ such that 
$\gcd(s,r)=1$. 

Clearly,   the bound~\eqref{eq:Bound W_A} also allows to study 
the distribution of primitive roots in the set $\cW_\fA$. Finally, 
we note that the case when the sets $\cA_1, \ldots, \cA_n$ are 
sets of consecutive residues corresponds to the settings 
of~\cite[Theorem~3.1]{DMS}. In this case,
one can use the recent generalisations
of the Burgess bounds that are due to Chang~\cite{Chang1,Chang2} 
and Konyagin~\cite{Kon2} (when $n$ is bounded) together with 
a classical bound of Davenport and Lewis~\cite{DaLe}
(when $n \to \infty$) and improve the result of~\cite[Theorem~3.1]{DMS}.

\section*{Acknowledgements}

The authors would like to thank Sergei Konyagin for helpful discussions and in particular for his suggestion to use the result of Shao [35] in the proof of Theorem~\ref{thm:Hilb}.

 The authors are also very grateful to
C\'{e}cile Dartyge and to the referee  for 
 careful reading of the manuscript and useful suggestions.

Most of this work originated at the Erwin Schr{\"o}dinger International Institute for Mathematical Physics,
Vienna, and then was continued during a very enjoyable stay of the authors
at the CIRM, Luminy. The second  author was partially supported by the FWF (Austria) grant W1230.
The third author was partially  supported by 
the CIRM-CNRS-SMF (France) as a Jean-Morlet Chair  and ARC (Australia) grant  
DP140100118.


\begin{thebibliography}{9999}

\bibitem{BaCoSh} W. Banks, A. Conflitti and I. E. Shparlinski,
`Character sums over integers with restricted
$g$-ary digits',   {\it Illinois J. Math.\/},
{\bf 46} (2002), 819--836.

\bibitem{BaSh1} W. D. Banks and I. E. Shparlinski, 
`Arithmetic properties of numbers with restricted digits', 
{\it Acta Arith.\/}, {\bf 112} (2004), 313--332. 

\bibitem{BaSh2} W. D. Banks and I. E. Shparlinski, 
`Average value of the Euler function on binary palindromes', 
{\it Bull. Pol. Acad. Sci. Math.\/}, {\bf 54} (2006),  95--101. 

\bibitem{Bour0}
J. Bourgain, `Estimates  on exponential sums related to Diffie-Hellman
distributions', {\it  Geom. and Funct. Anal.\/}, 
{\bf 15} (2005), 1--34.

\bibitem{Bour1} J. Bourgain,  
`Prescribing the binary digits of primes', 
{\it Israel J. Math.\/}, {\bf 194} (2013), 935--955. 

\bibitem{Bour2} J. Bourgain,  
`Prescribing the binary digits of primes, II', 
{\it Israel J. Math.\/}, {\bf 206} (2015), 165--182.

\bibitem{Chang1}
M.-C. Chang, `On a question of Davenport and Lewis
and  new character sum
bounds   in finite fields', 
{\it Duke Math. J.\/},  {\bf 145} (2008), 409--442. 

\bibitem{Chang2}
M.-C. Chang, `Burgess inequality in $\F_{p^2}$', 
{\it Geom. Funct. Anal.\/}, 
{\bf 19} (2009), 1001--1016.

\bibitem{Col1} S. Col, `Diviseurs des nombres ellips{\'e}phiques', 
{\it Period. Math. Hungar.\/}, {\bf 58} (2009), 1--23. 

\bibitem{Col2} S. Col, `Palindromes dans les progressions 
arithm{\'e}tiques', {\it Acta Arith.\/}, {\bf 137} (2009),  1--41. 

\bibitem{DMS} C. Dartyge, C. Mauduit and A. S{\'a}rk{\"o}zy, 
`Polynomial values and generators with missing digits in finite fields',
{\it Funct. Approx. Comment. Math.\/}, {\bf 52}  (2015), 65--74.

\bibitem{DaLe}
H. Davenport and D. J. Lewis,  
`Character sums and primitive roots in finite fields',
{\it Rend. Circ. Mat. Palermo\/}, {\bf 12} (1963), 129--136. 

\bibitem{DiDaSiHa}  J. A. Dias da Silva and  Y. O. Hamidoune,
`Cyclic spaces for Grassmann derivatives and additive theory', 
{\it Bull. London Math. Soc.\/}, {\bf 26} (1994) 140--146.

\bibitem{DES1} R. Dietmann, C. Elsholtz and I. E.~Shparlinski,
`On gaps between primitive roots in the Hamming metric', 
{\it Quart. J. Math.\/}, {\bf 64} (2013), 1043--1055.

\bibitem{DES2} R. Dietmann, C. Elsholtz and I. E.~Shparlinski,
`On gaps between quadratic non-residues in the Euclidean and Hamming metrics',
{\it Indag. Math.\/}, {\bf 24} (2013), 930--938.

\bibitem{DrMa} M. Drmota and C. Mauduit, `Weyl sums over integers with 
affine digit restrictions',
{\it J. Number Theory\/}, {\bf 130} (2010),  2404--2427.

\bibitem{DMR} M. Drmota, C. Mauduit and J. Rivat, 
`Primes with an average sum of digits', 
{\it Compos. Math.\/}, {\bf 145} (2009),  271--292.

\bibitem{EMS} P. Erd{\H o}s, C. Mauduit and A. S{\'a}rk{\"o}zy,
`On arithmetic properties of integers with missing digits, II',
{\it Discr. Math.\/}, {\bf 200} (1999), 149--164.

\bibitem{GraShp} S. W. Graham and I. E. Shparlinski, 
`On RSA moduli with almost half of the bits prescribed',
{\it Discrete Applied Math.\/}, {\bf 156} (2008), 3150--3154.


\bibitem{GrLoSch} M. Gr{\"o}tschel, L. Lov{\'a}sz and A. Schrijver,
{\it Geometric algorithms and combinatorial optimization\/},
Springer-Verlag, Berlin, 1993.

\bibitem{HardyWright} G. H. Hardy and E. M. Wright, {\it An introduction to
the theory of numbers\/}, Oxford Univ. Press, Oxford, 1979.

\bibitem{HaKa}
G. Harman and I. Katai, `Primes with preassigned digits II', 
{\it Acta Arith.\/}, {\bf 133} (2008), 171--184. 


\bibitem{HS} N. Hegyv\'{a}ri and A. S\'{a}rk\"{o}zy,  
`On Hilbert cubes in certain sets', 
{\it Ramanujan J.\/}, {\bf 3} (1999), 303--314.



 \bibitem{IwKow} H. Iwaniec and E. Kowalski,
{\it Analytic number theory\/}, Amer.  Math.  Soc.,
Providence, RI, 2004.

\bibitem{Kar1} A. A. Karatsuba, `The distribution of values
of  Dirichlet characters on additive sequences', {\it Doklady
Acad. Sci. USSR\/}, {\bf 319} (1991),   543--545 (in   Russian).

\bibitem{Kar2} A. A. Karatsuba, {\it Basic analytic number theory\/}, 
Springer-Verlag, 1993.


\bibitem{Kon1} 
S. V. Konyagin, `Arithmetic properties of integers with missing digits: distribution in residue classes', 
{\it Period. Math. Hungar.\/}, {\bf 42} (2001), 145--162. 

\bibitem{Kon2} S. V. Konyagin, 
`Bounds of character sums in finite fields',  
{\it Matem. Zametki\/}, {\bf 88} (2010),  529--542 (in Russian). 

\bibitem{KMS} 
S. V. Konyagin, C. Mauduit and A. S{\'a}rk{\"o}zy, `On the number of prime factors of integers characterized by digit properties', 
{\it Period. Math. Hungar.\/}, {\bf 40} (2000), 37--52. 

\bibitem{Luca}  F. Luca, `Arithmetic properties of positive integers with
fixed digit sum', {\it Revista Matem. Iberoamer.\/},  
{\bf 22} (2006),  369--412. 


\bibitem{MauRiv}  C. Mauduit and J. Rivat, 
`Sur un probl{\'e}me de Gelfond: La somme des chiffres des nombres premiers', 
{\it Ann. of Math.\/}, {\bf 171}   (2010),  1591--1646. 

\bibitem{MaSha}  C. Mauduit and Z. Shawket,
`Sommes d'exponentielles associ{\'e}es aux fonctions digitales restreintes', 
{\it Unif. Distrib. Theory\/}, {\bf 7} (2012), 105--133. 

\bibitem{MoShk}  N. G. Moshchevitin and I. D.  Shkredov, 
`On the multiplicative properties modulo $m$ of 
numbers with missing digits', 
{\it Mat. Zametki (Transl. as  Math. Notes)\/}, {\bf 81} (2007), 385--404
(in Russian). 


\bibitem{OstShp} A. Ostafe and I.~E.~Shparlinski,  
`Multiplicative character sums and products of sparse integers in residue classes', 
{\it Period. Math. Hungarica\/},  {\bf 64} (2012), 247--255.

\bibitem{Shao}
X. Shao, `Character sums over unions of intervals',
{\it  Forum Math.\/},  {\bf 27} (2015), 3017--3026.


\bibitem{Sch} T. Schoen,  
`Arithmetic progressions in sums of subsets of sparse sets',
{\it Acta Arith.\/}, {\bf 147} (2011), 283--289.


\bibitem{Shp1} I. E. Shparlinski, 
`Prime divisors of sparse integers',  
{\it Period. Math. Hungarica\/}, {\bf 46}  (2003),  215--222.

\bibitem{Shp2}
I.~E.~Shparlinski,  `On RSA moduli with prescribed bit patterns',
{\it Designs, Codes and Cryptography\/}, {\bf 39} (2006),
113--122.

\bibitem{Shp3} I. E. Shparlinski, 
`Exponential sums and prime divisors',  
{\it Period. Math. Hungarica\/}, {\bf 57}  (2008),    93--99.


\end{thebibliography}
\end{document}